\documentclass{amsart}
\usepackage{amsmath,amsfonts} 
\usepackage{amsthm} 
\usepackage{amssymb} 
\usepackage[T1]{fontenc} 
\usepackage[utf8]{inputenc}
\usepackage{textcomp}
\usepackage[dvips]{graphicx}
\usepackage{verbatim}
\usepackage{todonotes}
\usepackage{mathrsfs}
\usepackage[english]{babel}
\usepackage[babel]{csquotes}
\usepackage{xy}
\xyoption{all}
\usepackage{hyperref}
\usepackage[style=alphabetic,hyperref,backend=bibtex]{biblatex}
\bibliography{alggeo}

\theoremstyle{definition}
\newtheorem{definition}{Definition}[section]
\theoremstyle{plain}
\newtheorem{theorem}[definition]{Theorem}

\newtheorem{innercustomthm}{Theorem}

\newtheorem*{conj}{Conjecture}
\newtheorem{proposition}[definition]{Proposition}
\newtheorem{lemma}[definition]{Lemma}
\newtheorem{corollary}[definition]{Corollary}
\theoremstyle{remark}
\newtheorem{remark}[definition]{Remark}
\newtheorem{example}[definition]{Example}

\DeclareMathOperator{\Ext}{Ext}
\DeclareMathOperator{\Hom}{Hom}

\DeclareMathOperator{\Spec}{Spec}

\DeclareMathOperator{\rs}{rs}
\DeclareMathOperator{\Strat}{Strat}

\DeclareMathOperator{\Repf}{Repf}

\DeclareMathOperator{\Vtf}{Vtf}

\DeclareMathOperator{\id}{id}

\DeclareMathOperator{\str}{str}

\DeclareMathOperator{\red}{red}
\DeclareMathOperator{\tp}{top}
\DeclareMathOperator{\Der}{Der}
\DeclareMathOperator{\et}{\acute{e}t}

\newcommand{\Oh}{\mathscr{O}}

\newcommand{\D}{\mathscr{D}}

\newcommand{\bin}[2]{\genfrac{(}{)}{0pt}{}{#1}{#2}}

\begin{document}

\title{Gieseker conjecture for homogeneous spaces}
\author{Giulia Battiston}
\date{\today}
\address{ Ruprecht-Karls-Universität Heidelberg -Mathematisches Institut -
Im Neuenheimer Feld 288
D-69120 Heidelberg }
\thanks{This work was supported by DFG-Forschergruppe 1920}
\subjclass[2010]{14H30, 14G17, 13N10}
\email{gbattiston@mathi.uni-heidelberg.de}

\begin{abstract}
We prove Gieseker conjecture for an homogeneous space $X$, saying that if $X$ has no non-trivial tame coverings then it has no non-trivial regular singular $\Oh_X$-coherent $\D_{X/k}$-modules. In order to do so we prove a K\"unneth formula for the regular singular stratified fundamental group and a base change for Gauss-Manin stratifications in the non-proper case.
\end{abstract}
\maketitle

It is a classical result that the topological fundamental group satisfies a K\"unneth formula: namely, if $(X,x)$ and $(Y,y)$ are topological spaces then $\pi_1^{\tp}(X\times Y,x\times y)\simeq \pi_1^{\tp}(X,x)\times \pi_1^{\tp}(Y,y)$. In particular, if $X$ and $Y$ are algebraic complex varieties, the same holds for the \'etale fundamental group $\pi^{\et}(X\times Y)$.

On the other hand, this is known to fail if the ground field $k$ has positive characteristic (unless one of the two varieties is proper) even for $\mathbb{A}^1_k\times \mathbb{A}^1_k$, due to the existence of \'etale covers that are wildly ramified at infinity. Still, if both varieties are smooth and admit good compactifications, it is proven in \cite[Prop.~3]{Hoshi}, in the language of log-schemes, that the tame fundamental group $\pi^{t}$ (as defined in \cite{Kerz/tameness}) satisfies a K\"unneth formula. The first goal of this article is to prove the same for the regular singular fundamental group $\pi^{\rs}$, that is the algebraic $k$-group scheme associated with the Tannaka category of regular singular stratified bundles (see Section~\ref{sec:rsbd} for definitions):
\begin{innercustomthm}[see {Prop.~\ref{prop:Kun}}]\label{thm:1} Let $X$ and $Y$ be two smooth connected varieties over  an algebraically closed field $k$ and fix $x\in X(k)$ and $y\in Y(k)$. Then the canonical morphism
\[\pi^{\rs}(X\times Y,x\times y)\to \pi^{\rs}(X,x)\times \pi^{\rs}(Y,y)\]
is an isomorphism.
\end{innercustomthm}

As the tame fundamental group is the profinite completion of the regular singular one (see \cite[Cor.~5.6]{Kindler/FiniteBundles}), this gives in particular a new proof of the K\"unneth formula for the tame fundamental group, independent from the existence of good compactifications.

\

The second goal of this paper is to apply the previous theorem to a conjecture of Gieseker, analogous to Grothendieck and Mal\v{c}ev's theorem (see \cite{Grothendieck/Malcev}): let $k$ be an algebraically closed field of positive characteristic and $X$ a smooth variety over $k$, then Gieseker conjectured in \cite{Gieseker/FlatBundles} that $\pi^{\Strat}(X)=0$ if and only if $\pi^{\et}(X)=0$, where $\pi^{\Strat}(X)$ is the stratified fundamental group, that is the $k$-group scheme corresponding via Tannaka duality to the category of $\Oh_X$-coherent $\D_{X/k}$-modules. It was proven by Esnault and Mehta in \cite{Esnault/Gieseker} that if $X$ is proper this is indeed the case.
If $X$ is not proper, though, the following variant of the conjecture was proposed in \cite[Q.~3.1]{Esnault/ECM} and further discussed in \cite[Conj.~1.3]{Kindler/Evidence}:
\begin{conj} Let $X$ be a smooth variety over an algebraically closed field $k$ of positive characteristic. Then $\pi^{t}(X)=0$ if and only if $\pi^{rs}(X)=0$.
\end{conj}
As the tame fundamental group is the profinite completion of the regular singular, one direction is trivial. In this article we prove the other direction for \emph{quasi-homogeneous spaces}, that is varieties admitting an homogeneous space as dense open subset:
\begin{innercustomthm}[{see Thm.~\ref{thm:conj}}]\label{thm:2}
Let $X$ be a smooth connected quasi-homogeneous space over a positive characteristic field $k$ and let $x\in X(k)$. Let $\bar{k}$ be an algebraic closure of $k$, then $\pi^{t}(X\times_{\Spec k}\Spec\bar{k})=0$ if and only if $\pi^{rs}(X,x)=0$.
\end{innercustomthm}

Theorem~\ref{thm:2} relies on Theorem~\ref{thm:1} as follows: by a classical argument of Grothendieck, the fact that $\pi^{\rs}$ satisfies the K\"unneth formula implies $\pi^{\rs}(G)$ is commutative when $G$ is a connected group scheme. Let $G\to X=G/H$ be the quotient map to the connected homogeneous space we are considering. Then, without loss of generality, $G$ can be assumed to be connected and we prove that $\pi^{\rs}(X)$ is an extension of a quotient of $\pi^{\rs}(G)$ by a finite group scheme. In particular the hypothesis of the theorem implies that $\pi^{\rs}(X)$ is commutative as well and we can apply a slight strengthening of the results from \cite{Kindler/Evidence} to conclude.

Similar methods also give a relative version of the theorem (see Proposition~\ref{prop:giemor}): if $f:X\to Y$ is a morphism of proper homogeneous spaces inducing an isomorphism on the tame fundamental groups, the same holds for the regular singular ones.

\

As for the K\"unneth formula, its proof highly relies on (and is actually equivalent to the good properties of) the so called $0$-th Gauß-Manin stratification, introduced by Ph\`ung in \cite{Phung}. It is a positive characteristic analogue of the Gauß-Manin connection, and Ph\`ung uses it to give a (new) proof of the K\"unneth formula for $\pi^{\Strat}(X\times Y)$ when $X$ is proper over $k$. For our purposes, the main property of the $0$-th Gauß-Manin stratification is that it behaves well under the base change to the fibers of proper morphisms. 
We prove that the same holds more generally for smooth morphisms $f:X\to S$ with smooth relative compactification $\overline{X}$ over $S$, if we start with a regular singular stratified bundle (see Prop.~\ref{prop:bc}). The crucial point in proving this is to understand the maximal sub-object of a regular singular stratified bundle that extends to an actual stratified bundle on some smooth compactification, and to prove in Proposition~\ref{prop:max} that this maximal object behaves well under the base change to the fibers of smooth morphisms.

The last ingredient is a study of the behaviour of the stratified and regular singular fundamental group under a finite \'etale Galois cover (Lemma~\ref{lem:galois}), which was already proven in characteristic zero in \cite[§1.4]{Katz/Calculation}: if $X\to Y$ is a Galois \'etale cover of group $G$, then the sequence
\[1\to\pi^{\Strat}(X)\to\pi^{\Strat}(Y)\to G\to 1\]
is exact and the sequence
\[1\to\pi^{\rs}(X)\to\pi^{\rs}(Y)\to G^t\to 1,\]
is exact on the right, where $G^t$ is the maximal quotient of $G$ corresponding to a tame cover of $Y$. Moreover it is exact on the left whenever $G=G^t$.

The article is organized as follows. In Section~\ref{sec:rsbd} we introduce regular singular stratified bundles and we recall their main properties, in Section~\ref{sec:repthy} we collect a few results in representation theory that will be used in the various proofs. In Section~\ref{sec:max} we prove the base change for the maximal sub-stratified bundle, while in Section~\ref{sec:GM} we prove the same property for the Gauss-Manin connection associated with a regular singular stratified bundle. In Section~\ref{sec:galois} we study the behaviour of the category of stratifications under a fintite \'etale Galois cover and Section~\ref{sec:Kun} is dedicated to prove the K\"unneth formula for the regular singular fundamental group.  Finally, the proof of Gieseker conjecture for quasi-homogeneous spaces occupies Section~\ref{sec:conj}.

\subsection*{Acknowledgements} I would like to thank Wolfgang Soergel, Fritz H\"ormann and  Lars Kindler for helpful discussions and H\'el\`ene Esnault for useful comments on a first version of this article.

\subsection*{Notation and general definitions} When not otherwise specified, fiber products are taken over the spectrum of the ground field $k$, which will always be assumed to be of positive characteristic.  A variety is a separated scheme of finite type over $k$. 

When the base scheme is the spectrum of a field, we will often abbreviate relative notions over $\Spec k$ writing only $k$ instead: for example $\D_{X/k}$ will stand for $\D_{X/\Spec k}$ and so on.

If $X$ is a smooth $S$-variety, a \emph{partial good compactification} $(X,\overline{X})$ of $X$ over $S$ is an open $S$-embedding $X\subset \overline{X}$  in a variety smooth over $S$, such that $\overline{X}-X$ is a snc divisor. We say that $(X,\overline{X})$ is a \emph{good compactification} over $S$ if moreover $\overline{X}$ is proper over $S$. When the base $S$ is not specified, we mean over the spectrum of the ground field.
\section{Regular singular stratified bundles}\label{sec:rsbd}
We collect in this section the definitions and the properties of the main objects of study of this article, that is regular singular stratified bundles and the group scheme associated to them via Tannaka duality (\cite[Thm.~2.11]{DeligneMilne}).
If $X$ is a smooth variety, then the sheaf of differential operators $\D_{X/k}$, as defined in \cite[§16]{SGA4}, is locally generated over $\Oh_X$ by the operators $\partial_{x_i}^{(k)}$, $k\in\mathbb{N}$, where we define
\[\partial_{x_i}^{(k)}(x_j^h)=\delta_{ij} \bin{h}{k} x_j^{h-k},\]
with $x_1,\dotsc, x_d$ denoting a choice of local \'etale coordinates for $X$.

If $(X,\overline{X})$ is a partial good compactification with $D=(\overline{X}\backslash X)_{\red}$, then $\D_{\overline{X}/k}(log D)$, denotes is the subsheaf of $\D_{\overline{X}/k}$ consisting of all those operators locally fixing all powers of the defining ideal of $D$. If locally $D$ is defined by $x_1\cdot\dotsc\cdot x_l=0$, then $\D_{\overline{X}/k}(log D)$ is generated by the operators $x_i^k\partial_{x_i}^{(k)}$ for $i=1,\dotsc,l$ and $\partial_{x_j}^{(k)}$ for $i=l+1,\dotsc,d$ and $k\in\mathbb{N}$. For $\underline{n}=(n_1,\dotsc, n_d)$, we denote
\[\mathsf{D}_{\underline{n}}=\partial_{x_1}^{(n_1)}\dotsc\partial_{x_d}^{(n_d)}.\]

If $X\to S$ is a smooth morphism, and $X\subset\overline{X}$ is a partial good compactification over $S$, we set $\D_{X/S}$ to be the subsheaf of $\D_{X/k}$ of operators that are $\Oh_S$-linear and $\D_{\overline{X}/S}(log D)=\D_{\overline{X}/k}(log D)\cap \D_{X/S}$.

The sheaf $\D_{X/k}$ has a filtration $\D_{X/k}^{<i}$ given locally by differential operators of order strictly less than $p^i$, where $p$ is the characteristic of $k$, that is $\D_{X/k}^{<0}=\Oh_X$, $\D_{X/k}^{<1}$ is the span of $\Der(X/k)$ and so on.

\begin{definition} Let $X$ be a smooth variety over $k$.
\begin{itemize}
\item[i)]A \emph{stratified bundle} over $X$ is a $\D_{X/k}$-module which is $\Oh_X=\D_{X/k}^{< 0}$-coherent. If $X\to S$ is a smooth morphism of smooth varieties, a \emph{relative} stratified bundle is a $\Oh_X$-coherent $\D_{X/S}$-module. Morphism of (relative) stratified bundles are \emph{horizontal} morphisms, that is commuting with the action of the differential operators. We denote by $\Strat(X/k)$ the category of stratified bundles over $X$ and by $\Strat(X/S)$ the category of relative stratified bundles.
\item[ii)]Let $(X,\overline{X})$ be a partial good compactification and let $D=(\overline{X}\backslash X)_{\red}$, then a $\D_{X/k}$-module is called \emph{$(X,\overline{X})$-regular singular} if it admits a torsion free $\Oh_{\overline{X}}$-coherent $\D_{\overline{X}/k}(\log D)$-extension. We say that a $\D_{X/k}$-module is \emph{regular singular} if it is $(X,\overline{X})$-regular singular for every partial good compactification $(X,\overline{X})$. We denote the full subcategory of $\Strat(X/k)$ of $(X,\overline{X})$-regular singular stratified bundles by $\Strat^{\rs}(X,\overline{X})$ and the one of regular singular stratified bundles by $\Strat^{\rs}(X/k)$.
\end{itemize} \end{definition}

If $\mathcal{T}$ is a Tannaka category over $k$ and $T\in\mathcal{T}$, we denote by $\langle T\rangle_\otimes$ the minimal full subcategory of $\mathcal{T}$ containing $T$ and closed under direct sums, duals, tensor product and subobjects and if $\omega:\mathcal{T}\to \Vtf_k$  is a neutral fiber functor, we denote by $\pi(T,\omega)$ the Tannaka dual of $(\langle T\rangle_\otimes,\omega)$.

We collect here a list of properties of the objects we just defined, that will be needed along the article:
\begin{proposition}\label{prop:prop} Let $X$ be a smooth connected variety, $\overline{X}$ a partial good compactification and $x\in X(k)$.
\begin{itemize}
\item[i)]\emph{(see \cite{Saavedra} and \cite{Kindler/FiniteBundles})} $\Strat(X/k)$, $\Strat^{\rs}(X,\overline{X})$ and $\Strat^{\rs}(X/k)$ are Tannakian categories, and taking the fiber at the rational point $x$ neutralizes them. We denote by $\pi^{\Strat}(X,x)$, $\pi^{\rs}(X,\overline{X},x)$ and $\pi^{\rs}(X,x)$ their associated Tannaka duals.
\item[ii)]\emph{(see \cite{Kindler/FiniteBundles})} if $E\in\Strat(X/k)$, $U$ is a non-empty open of $X$ and $x\in U(k)$ then the restriction functor induces an isomorphism $\pi(E_{\mid U})\simeq \pi(E)$ and faithfully flat maps $\pi^{\Strat}(U,x)\to\pi^{\Strat}(X,x)$ and $\pi^{\rs}(U,x)\to\pi^{\rs}(X,x)$.
\item[iii)]\emph{(see \cite{Kindler/FiniteBundles})} if $U\subset \overline{X}$ is an open whose complement has codimension at least $2$, then the restriction functor induces an equivalence of categories $\Strat(X/k)\to\Strat(X\cap U/k)$ and $\Strat^{\rs}(X,\overline{X})\to\Strat^{\rs}(X\cap U,U)$. 
\end{itemize}
\end{proposition}

Let $f:X\to S$ be a smooth morphism of smooth varieties, and let $X^{(i)}=X\times_{S,F_S^i}S$ be the $i$-th Frobenius pullback, where $F_S$ is the absolute Frobenius on $S$. We will denote by $F_{X^{(i)}/S}:X^{(i)}\to X^{(i+1)}$ the relative Frobenius, given by the structure morphism to $S$ and the absolute Frobenius on $X^{(i)}$.

For any $E\in\Strat(X/S)$ we denote by
\[\bar{E}^{(i)}=\{e\in \bar{E}\mid \mathsf{D}(e)=0 \quad\forall \mathsf{D}\in\mathcal{D}_{\overline{X}/S}^{<i}\text{ s.t. }\mathsf{D}(1)=0\},\]
which has a natural structure of $X^{(i)}$-module, moreover by Cartier theorem (see \cite[Thm.~5.1]{Katz/Connections}) the natural map $\sigma_i:F_{X^{(i-1)}/S}^*E^{(i)}\to E^{(i-1)}$ is an isomorphism. The next theorem, fundamentally due to Katz, can be found in \cite{Gieseker/FlatBundles} for the absolute case and \cite{Phung} for the relative one:

\begin{theorem}\label{thm:katz} The assignment $E\mapsto \{E^{(i)},\sigma_i\}$ induces an equivalence of categories  between $\Strat(X/S)$ and the category whose objects are collections $\{E^i,\sigma_i\}$, with $E^i$ $\Oh_{X^{(i)}}$-coherent modules and $\sigma_i:F_{X^{(i-1)}/S}^*E^i\simeq E^{i-1}$, and whose morphisms $\{f_i\}:\{E^i,\sigma_i\}\to\{F^i,\tau_i\}$ are collection of  morphism of $\Oh_{X^{(i)}}$-modules $\{f^i:E^i\to F^i\}$ such that $\tau_i\circ F^*_{X^{(i)}/S}f_{i+1}=f_i\circ \sigma_i$.
\end{theorem}

\section{A little bit of representation theory}\label{sec:repthy}
Let $L\xrightarrow{a} G\xrightarrow{b} H$ be a sequence of affine group schemes over $k$, and let
\[\Repf_k(H)\xrightarrow{B}\Repf_k(G)\xrightarrow{A}\Repf_k(L)\]
be the induced functors on the categories of finite dimensional representations over $k$. We recall here a very useful criterion to detect from the latter exacntess properties of the first:
\begin{proposition}\label{prop:crit} Keeping the same notation as above, the following holds:
\begin{itemize}
\item[i)]\emph{(see \cite[Prop.~2.21,a]{DeligneMilne})} The map $b$ is faithfully flat if and only if the functor $B$ is fully faithful and its essential image is closed under subobjects in $\Repf_k(G)$.
\item[ii)]\emph{(see \cite[Prop.~2.21,b]{DeligneMilne})} The map $a$ is a closed immersion if and only if the essential image of $A$ spans $\Repf_k(L)$, that is for every $V\in\Repf_k(L)$ there exists $W\in\Repf_k(G)$ such that $V\in\langle A(W)\rangle_\otimes$.
\item[iii)]\emph{(see \cite[Thm.~A1,iii]{EsnaultSun})} If $b$ is faithfully flat and $a$ is a closed immersion, then the sequence $L\xrightarrow{a} G\xrightarrow{b} H$ is exact if and only if the following conditions hold:
\begin{itemize}
\item[a)] $\ker a$ is the normal hull of $L$, that is for every $V\in\Repf_k(G)$ we have that $A(V)$ is trivial if and only if $V$ is in the essential image of $B$.
\item[b)] for every $V\in\Repf_k(G)$ if $A(V)^0\subset A(V)$ denotes it maximal trivial subrepresentation, then there exists $V'\subset V$ so that $A(V)^0\simeq A(V')$.
\item[c)] for every $V\in\Repf_k(L)$ there exists $W\in\Repf_k(G)$ such that $V\subset A(W)$ or such that $V$ is a quotient of $A(W)$.

\end{itemize}
\end{itemize}

\end{proposition}

\section{The maximal sub-stratified bundle}\label{sec:max}
Let $\bar{E}$ be a $\Oh_{\overline{X}}$-coherent $\D_{\overline{X}/S}(\log D)$-module on some smooth variety $\overline{X}$ with log structure given by a snc divisor $D$ and with $\overline{X}\to S$ a smooth morphism. We want to find a useful description of the maximal subobject of $\bar{E}$ which is actually a $\D_{\overline{X}/S}$-module.
Similarly as in Theorem~\ref{thm:katz} we define 
\[\bar{E}^{(i)}=\{e\in \bar{E}\mid \mathsf{D}(e)=0 \quad\forall \mathsf{D}\in\mathcal{D}_{\overline{X}/S}^{<i}(\log D)\text{ s.t. }\mathsf{D}(1)=0\},\]
then they are naturally $\Oh_{X^{(i)}}$-modules. Note that if $\bar{E}$ is torsion free then $E^{(i)}$ is torsion free for every $i$. 

We will use the following notation: if $k\leq i\in \mathbb{N}$ we will denote
\[F_{X/S,i}^{*,k}=F_{X^{(i-k)}/S}^*\circ\dotsc \circ F_{X^{(i-2)}/S}^*\circ F_{X^{(i-1)}/S}^*.\]
\begin{lemma}\label{lem:inj} Assume that $\bar{E}$ is torsion free, then for every $k\leq i$, the natural morphism $F_{X/S,i}^{*,k}(\bar{E}^{(i)})\to \bar{E}^{(i-k)}$ is injective. 
\end{lemma}
\begin{proof} It is enough to do it for $k=i=1$, as for smooth morphisms the relative Frobenius is faithfully flat.
Let $j:X\subset \overline{X}$ be the immersion of the open where the log structure is trivial, and similarly for $i:X^{(1)}\subset \smash{\overline{X}}^{(1)}$. The natural map $F_{\overline{X}/S}^*(\bar{E}^{(1)})\to \bar{E}$ sits in the commutative diagram
\[
\xymatrix{
F_{\overline{X}/S}^* (\bar{E}^{(1)})\ar[r]\ar[d]& \bar{E}\ar[d]\\
j_* j^* F_{\overline{X}/S}^* (\bar{E}^{(1)})=F_{\overline{X}/S}^*i_*i^*(\bar{E}^{(1)})\ar[r] & j_*j^*\bar{E}
}.\]
Now, $j^*F_{X/S}^*=F_{U/S}^*i^*$ and $j^*\bar{E}=E$ is a stratified bundle, hence by Theorem~\ref{thm:katz}
\[j^*F_{X/S}^*(\bar{E}^{(1)})=F_{U/S}^* E^{(1)}\to j^* \bar{E}=E\]
 is an isomorphism, and so is the lower line of the diagram. Hence it is enough to prove that 
\[\bar{E}^{(1)}\to i_*i^* \bar{E}^{(1)}\]
is injective, which follows from the fact that $\bar{E}^{(1)}$ is torsion free.
\end{proof}

Let $\bar{E}$ be a $\D_{\overline{X}/S}(\log D)$-module. By the previous lemma we can identify $F_{X/S,i}^{*,k}(\bar{E}^{(i)})$ with its image in $\bar{E}^{(i-k)}$. We define the stratified bundle $\bar{E}_m$ in the following way: 
\[\bar{E}_m^{(i)}=\bigcap_{k\in\mathbb{N}} F_{X/S,i+k}^{k,*}(\bar{E}^{(i+k)})\subset E^{(i)}.\]
By Lemma~\ref{lem:inj}, it is clear that the natural morphism $\sigma_i:F_{X/S}^*\bar{E}_m^{(i)}\to \bar{E}_m^{(i-1)}$ is an isomorphism, therefore by Theorem~\ref{thm:katz} $\{\bar{E}_m^{(i)},\sigma_i\}$ gives a stratified bundle $\bar{E}_m$ on $\bar{X}$, that we call the \emph{maximal sub-stratified bundle} of $\bar{E}$. The definition is justified by the following

\begin{lemma}\label{lem:max}
Let $\bar{E}_m\subset \bar{E}$ defined as above, then $\bar{E}_m$ is the maximal sub-$\D_{\bar{X}/S}$-module of $\bar{E}$, that is the maximal subobject on which the $\D_{\bar{X}/S}(\log D)$-action extends to a $\D_{\bar{X}/S}$-action. In particular for every open $V\subset \overline{X}$ containing the generic points of every component of $D$, we have that
\[(\bar{E}_m)_{|V}=(\bar{E}_{|V})_m.\]

\end{lemma}
\begin{proof}
The last part of the theorem follows from Proposition~\ref{prop:prop} $(iii)$.

For the first, let $\bar{E}'\subset \bar{E}$ be any sub-$\D_{\bar{X}/S}(\log D)$-module such that the action extends to a $\D_{\bar{X}/S}$-module action. 
A priori there are two different sets:
\[\bar{E}'^{(i)}=\{e\in \bar{E}'\mid \mathsf{D}(e)=0 \quad\forall \mathsf{D}\in\mathcal{D}_{\overline{X}/S}^{<i}(\log D)\text{ s.t. }\mathsf{D}(1)=0\}\]
and 
\[\bar{E}'^{(i)}_{\D}=\{e\in \bar{E}'\mid \mathsf{D}(e)=0 \quad\forall \mathsf{D}\in\mathcal{D}_{\overline{X}/S}^{<i}\text{ s.t. }\mathsf{D}(1)=0\},\]
but if $\mathsf{D}\in\mathcal{D}_{\overline{X}/S}^{<i}$, then locally there exists a $\Oh_{\bar{X}}$ section $a$ such that $a\mathsf{D}\in\mathcal{D}_{\overline{X}/S}^{<i}(\log D)$. Hence if $e\in\bar{E}'^{(i)}$, we have that $(a\mathsf{D})(e)=a (\mathsf{D}(e))=0$, hence $\mathsf{D}(e)=0$. In particular, $\bar{E}'^{(i)}=\bar{E}_{\D}'^{(i)}$.
By Cartier theorem (\cite[Thm.~5.1]{Katz/Connections}) for every $k\leq i\in\mathbb{N}$
\[F_{X/S,i}^{k,*}(\bar{E}'^{(i)})\to \bar{E}'^{(i-k)}\]
is an isomorphism and, by the very definition of $\bar{E}_m$, one has that $\bar{E}'^{(i)}\subset\bar{E}_m^{(i)}$, and $\bar{E}'\subset \bar{E}_m$.
\end{proof}

We want now to switch the prospective, that is we want to consider a regular singular stratified bundle $E$ on $X$ and understand its maximal subobject $E_m$ that extends to a stratified bundle on $\overline{X}$.
\begin{remark}
To be completely formal, one should write $E_{m,\overline{X}}$, or assume that $\overline{X}$ is proper, as the stratified bundle $E_m$ depends indeed on the partial compactification $X\subset \overline{X}$ we choose. Nevertheless, as we have fixed $\overline{X}$ we will keep the notation as simple as possible.
\end{remark}

 It is immediate to see that if $\bar{E}$ is a logarithmic extension of $E$ to $\overline{X}$ then $E_m$ is equal to the restriction to $X$ of $\bar{E}_m$.

Our goal is now to prove that assigning $E\mapsto E_m$ and restricting to the closed fibers of the morphism $f:X\to S$ commute. That is, we want to prove that for every $s\in S(k)$ and $E\in\Strat^{\rs}(X/S)$  the natural inclusion $(E_m)_s\subset (E_s)_m$ is actually an isomorphism. In order to do this we need to briefly review the theory of the exponents of a regular singular stratified bundle, developed in \cite{Gieseker/FlatBundles} and in \cite{Kindler/FiniteBundles}.

Let hence $E$ be a regular singular stratified bundle and assume that $\overline{X}=\Spec A$, $X=\Spec A[x_1^{-1}]$ and $D=\{x_1=0\}$ and that we can complete $x_1$ to a system of \'etale coordinates $x_1,\dots,x_n$. Then (see \cite[Lemma~3.8]{Gieseker/FlatBundles} or \cite[Prop.~3.10]{Kindler/FiniteBundles}) if $\bar{E}$ is any torsion free coherent logarithmic extension of $E$, then $\bar{E}_{|D_1}$ decomposes as a direct sum 
\[\bigoplus_{\alpha\in \mathbb{Z}_p}V_{\alpha}\]
such that every operator in the kernel of $i^*\D_{\overline{X}/k}(\log D)\to i^*\D_{\overline{X}/k}$ acts as multiplication by $\bin{\alpha}{n}$ on $V_{\alpha}$. Otherwise said, there exist elements $e_1,\dotsc e_r$ in $\overline{E}$ such that their restriction to $D$ spans $\bar{E}_{\mid D}$ and such that 
\[x_1^n\partial_{x_1}^{(n)}(e_i)=\bin{\alpha_i}{n} e_i + x_1\cdot \overline{E},\]
 for some $\alpha_i\in\mathbb{Z}_p$. The $\alpha_i$  are called the \emph{exponents} of $\bar{E}$ along $D$, they are well defined elements in $\mathbb{Z}_p$ and if $\bar{E}$ and $\bar{E}'$ are two different torsion free logarithmic extensions, their exponents differ by integers. In particular, we can talk about exponents of $E$ along $D$, as well defined elements in $\mathbb{Z}_p/\mathbb{Z}$. In general, if $D$ has different irreducible components, the exponents of $E$ along $D$ is the union of the exponents along the single irreducible components.
Moreover (see \cite[Thm.~5.2]{Kindler/FiniteBundles}), if $(X,\overline{X})$ is a partial good compactification and if $\tau:\mathbb{Z}_p/\mathbb{Z}\to\mathbb{Z}_p$ is a section of the quotient map, there exists a \emph{locally free} logarithmic extension $\bar{E}_\tau$ of $E$, called \emph{$\tau$-extension}, such that the exponents of $\bar{E}_\tau$ lie in the image of $\tau$ and which is unique up to isomorphisms that restrict to the identity on $E$.
\begin{remark} The previous theory does not make use of the full $\D$-module structure, but only of the operators that are in the kernel of $i^*\D_{\overline{X}/k}(\log D)\to i^*\D_{\overline{X}/k}$, which means that one can make sense of the whole theory also for the relative case, if all irreducible components of $D$ are dominant and smooth over the base $S$ and if $E\in\Strat^{\rs}(X/S)$ admits a coherent torsion free $\D_{\overline{X}/S}(\log D)$ extension. Similarly, the next lemmas can be extended to the relative case, with some extra conditions on the divisors at infinity.
\end{remark}


If the hypotheses of the next two results appear too technical, the reader may want to keep in mind that we will be mainly interested in the case where $X=Y\times S$ and $X\to S$ is the projection, while $\overline{X}=\overline{Y}\times S$ with $Y\subset\overline{Y}$ a good partial compactification (and then $s$ can be taken as any rational point of $S$).

\begin{lemma}\label{lem:exp} Let $f:X\to S$ be a smooth morphism of smooth connected varieties and let $E\in\Strat^{\rs}(X/k)$. Fix a partial good compactification $X\subset \overline{X}$ relative to $S$ and let $D=(\overline{X}\backslash X)_{\red}$. Fix an irreducible component $D'$ of $D$ and let $s\in S(k)$ such that $\overline{X}_s$ intersects the smooth locus of $D'$ non-trivially and such that $D'_s$ is a smooth divisor in $\overline{X}_s$. Let $\bar{E}$ be a torsion free $\Oh_{\overline{X}}$-coherent extension of $E$, then the exponents of $\overline{E}$ along $D'$ are exactly the exponents of $\bar{E}_s$ along $D'_s$. In particular, if $\overline{X}_s=\overline{X}\times_{S} s$ is a partial good compactification of $X_s$, $(\bar{E}_\tau)_s$ is a $\tau$-extension of $E_s$.
\end{lemma}
\begin{proof}
Let $\xi$ be the generic point of $D'_s$, then as $\overline{X}_s$ intersects the smooth locus of $D'$ there exists  an open neighbourhood $U\subset \bar{X}$ of $\xi$ such that $U=\Spec A$ and there are local \'etale coordinates $x_1,\dotsc x_n$ with $D\cap U=\{x_1=0\}$ (and hence $D'_s\cap U=\{\bar{x}_1=0\}$, where $\bar{x}_1$ is the restriction of $x_1$ to $\overline{X}_s$). Let $\overline{E}$ a torsion free logarithmic extension of $E$, then by the theory of exponents $\overline{E}_{|D\cap U}$ decomposes uniquely as a direct sum $\oplus_{\alpha\in\mathbb{Z}_p}V_\alpha$, but then 
\[\overline{E}_{|(D\cap U)_s}=(\overline{E}_{|D\cap U})_s=\bigoplus_{\alpha\in\mathbb{Z}_p}(V_\alpha)_s\]
which gives a decomposition for $\bar{E}_s$ along $D_s$ around $\xi$, hence the exponents have to coincide, because the action of $\bar{x}_1^n\partial_{\bar{x}_1}^{(n)}$ is determined by the one of $x_1^n\partial_{x_1}^{(n)}$.
\end{proof}

\begin{remark} We would like to point out the following subtlety: in \cite{Kindler/FiniteBundles} exponents are defined for locally free logarithmic extension (as these always exist, this is enough for the purposes of the paper). We, on the other hand, need the most generic definition of \cite{Gieseker/FlatBundles} (also detailed in \cite{Kindler/thesis}) because we will deal with torsion free rather than with locally free logarithmic extensions.
\end{remark}
We are now ready to prove the main result of this section.

\begin{proposition}\label{prop:max} Keeping the same notations and hypothesis of Lemma~\ref{lem:exp}, assume furthermore that $\overline{X}_s$ meets the smooth locus of every irreducible component of $D$. Then the natural inclusion $(E_m)_s\subset (E_s)_m$ is an isomorphism.
\end{proposition}
\begin{proof}
Note that the existence of the natural inclusion $(E_m)_s\subset (E_s)_m$ is due to the fact that $\overline{X}_s$ meets the smooth locus of all irreducible components of $D$.

Let's fix $\tau:\mathbb{Z}_p/\mathbb{Z}\to\mathbb{Z}_p$ a section of the projection and let  $\bar{E}$ be a $\tau$-extension of $E$. Then $\bar{E}^{(i)}$ has a natural structure of $\D_{X^{(i)}/k}(\log D^{(i)})$-module, which we can locally explicitly write  as follows: let $\mathsf{D}\in\D_{\overline{X}^{(i)}/S}(\log D^{(i)})$ and $e\in \bar{}^{(i)}$, then we can locally write $\mathsf{D}=\sum_{\underline{n}}a_{\underline{n}}\mathsf{D}_{\underline{n}}$, and we define
\[\mathsf{D}(e)\doteq  \sum_{\underline{n}}(a_{\underline{n}})^{p^i}\mathsf{D}_{p^i\cdot\underline{n}}(e),\]
where if $\underline{n}=(n_1,\dotsc, n_d)$, then $p^i\cdot\underline{n}=(p^in_1,\dotsc,p^in_d)$. It is clear that this action extends the $\Oh_{X^(i)}$-action, moreover if $\alpha_j$ are the exponents of $\bar{E}$ along $D$, then the exponents of $\bar{E}^{(i)}$ along $D^{(i)}$ are $\sigma_i(\alpha_j)$, where $\sigma_i:\mathbb{Z}_p\to\mathbb{Z}_p$ is defined as
\[\sigma_i(\sum a_lp^l)=\sum a_{l+i}p^l.\]
Summarizing, $\bar{E}^{(i)}$ is a torsion free coherent module whose exponents lie in the image of $\sigma_i\circ \tau$. In particular,  there is an open of $\overline{X}^{(i)}_s$ containing $X_s$ and all codimension $1$ points on which $(\bar{E}^{(i)})_s$ it is free and this, together with Lemma~\ref{lem:exp}, implies $(\bar{E}^{(i)})_s$ restricted to that open is a $\sigma_i\circ \tau$-extension of $E^{(i)}_s$. Similarly, as $\bar{E}_s$ is a $\tau$-extension of $E_s$, we have that $(\bar{E}_s)^{(i)}$ is a $\sigma_i\circ \tau$-extension of $\bar{E}^{(i)}$, over some open of $\overline{X}^{(i)}$ containing $X_s$ and all codimension $1$ points. Let $V_i$ be the intersection of the locally freeness locus of $(\bar{E}^{(i)})_s$ and $(\bar{E}_s)^{(i)}$, then restricted to $V_i$ both $(\bar{E}^{(i)})_s$ and $(\bar{E}_s)^{(i)}$ are $\sigma_i\circ \tau$ extensions and hence over $V_i$ there exists an isomorphism $\phi_i:(\bar{E}_s)^{(i)}\to (\bar{E}^{(i)})_s$ whose restriction to $E_s$ is the identity. 
Composing $\phi_i$ with the inclusion $(\bar{E}^{(i)})_s\subset(\bar{E}_s)^{(i)}$ we get a endmorphism $\rho_i:\bar{E}^{(i)}_s\to\bar{E}^{(i)}_s$ which, restricted to $E^{(i)}_s$ is the identity. In particular $\rho_i$ must be the identity as well and $\phi_i$ must be an isomorphism.
It follows that for every $i$ there is an open $V_i$ such that the inclusion $(\bar{E}_s)^{(i)}\subset(\bar{E}^{(i)})_s$ is an isomorphism. Without loss of generality, we can also assume that $F_{\overline{X}_s/k(s)}(V_i)\subset V_{i-1}$.
 
 Note that 
\[(\bar{E}_s)_m=\bigcap_{i>0} F_{\overline{X}_s/k(s),i}^{i,*}(\bar{E}_s)^{(i)},\]
while
\[(\bar{E}_m)_s=\big(\bigcap_{i>0}F_{\overline{X}/S,i}^{i,*}(\bar{E}^{(i)})\big)_s=\bigcap_{i>0}\big(F_{\overline{X}/S,i}^{i,*}(\bar{E}^{(i)})\big)_s=\bigcap_{i>0}F_{\overline{X}_s/k(s),i}^{i,*}(\bar{E}^{(i)})_s.\]

Now by way of contradiction assume that the inclusion $(\bar{E}_m)_s\subset (\bar{E}_s)_m$ is strict. Note that as the quotient $(\bar{E}_m)_s/(\bar{E}_s)_m$ is again a $\D_{\bar{X}_s/k(s)}$-module, it is locally free. Hence, if $e\in(\bar{E}_m)_s - (\bar{E}_s)_m$ then for every open $U\subset\bar{X}_s$ we have that (the localization of) $e$ must be in $((\bar{E}_m)_s)_{|U}- ((\bar{E}_s)_m)_{|U}$.
In particular we have that if $(\bar{E}_m)_s\subset (\bar{E}_s)_m$ is strict, then so is 
 $\lim_i((\bar{E}_m)_s)_{|V'_i}\subset\lim_i((\bar{E}_s)_m)_{|V'_i}$, where $V'_i\subset \overline{X}_s$ is the preimage of $V_i$ under the relative Frobenii. 
 
Let $W_i\subset\overline{X}$ such that $V_i=W_i\cap\overline{X}_s$ (in particular for every $i$, $W_i$ contains the generic points of all components of $D$), and let us chose them so that  $F^{-1}_{\overline{X}^{(i-1)}/S}(W_i)\subset W_{i-1}$, then by Lemma~\ref{lem:max} we have that 
\[((\bar{E}_m)_s)_{|V_i}=(\bar{E}_m)_{|V_i}=((\bar{E}_{|W_i})_m)_s=\bigcap_{i>0}F_{V_i/k(s),i}^{i,*}(\bar{E}^{(i)}_{|W_i})_s=\bigcap_{i>0}F_{V_i/k(s),i}^{i,*}(\bar{E}^{(i)})_{|V_i}\]
and
\[((\bar{E}_s)_m)_{|V_i}=((\bar{E}_s)_{|V_i})_m=\bigcap_{i>0} F_{V_i/k(s),i}^{i,*}(\bar{E}_s)_{|V_i}^{(i)},\]
and as $(\bar{E}_s)^{(i)}_{|V_i}=((\bar{E}^{(i)})_s)_{|V_i}=(\bar{E}^{(i)})_{|V_i}$, we get that 
\[\lim_i((\bar{E}_m)_s)_{|V_i}=\lim_i((\bar{E}_s)_m)_{|V_i},\]
and hence a contradiction.
\end{proof}

\begin{remark} It comes with no surprise that the previous proposition totally fails if $E$ is not regular singular. It is actually possible to construct an example of a $\D(\mathbb{A}^2/\mathbb{A}^1)$-module whose restriction to every fiber is trivial but which is not regular singular (see \cite{Battiston/pcur}).
\end{remark}

\section{The Gauß-Manin stratification for regular singular stratified bundles}\label{sec:GM}
In order to overcome the problem that Gauß-Manin connections do not in general extend to a stratification in positive characteristic, in \cite{Phung} Ph\`ung defined the $0$-th Gauß-Manin stratification relative to a smooth morphism $f:X\to S$ of smooth varieties. It is a functor $H^0_{\str}:\Strat(X)\to \Strat(S)$, defined in the following way: if $E\in\Strat(X)$, then 
\[H^0_{\str}(X/S,E)=\bigcap f_*E^{(i)}=f_*E^{\nabla_{X/S}},\]
where  $E^{\nabla_{X/S}}$ the sections of $E$ that are annihilated by all $\mathsf{D}\in\D_{X/S}$ with $\mathsf{D}(1)=0$.
One can prove (see \cite[Thm~2.2]{Phung}) that $f_*E^{\nabla_{X/S}}$ is $\Oh_S$-coherent and comes with a natural stratification. Moreover (see the proof of \cite[Lemma~1.5]{Phung}) the natural morphism $f^*f_*E^{\nabla_{X/S}}\to E$ is injective and horizontal. Finally (see \cite[Lemma~3.3]{Phung}) $f^*f_* E^{\nabla_{X/S}}$ is the maximal subobject of $E$ that is in the essential image of the pullback map $f^*:\Strat(S)\to\Strat(X)$ (in particular, if $S=\Spec k$, it is the maximal trivial subobject) and if $f$ is proper then for every $s\in S(k)$ we have that the natural inclusion
\[(f^*f_*E^{\nabla_{X/S}})_s\to f_s^*f_{s,*}^{}	 E_s^{\nabla_{X_s/k(s)}}\]
is an isomorphism.
\begin{remark} Note that even though in \cite[Lemma~3.3]{Phung} the map $f$ is assumed to be proper, it is immediate from the proof that the characterization of $f^*f_* E^{\nabla_{X/S}}$ as the maximal subobject of $E$ in the essential image of $f^*$ does not make use of the properness of $f$ and is true for any smooth map.
\end{remark}

\begin{proposition}\label{prop:bc} Let $X\to S$ be a smooth morphism of smooth varieties, and let $X\subset \overline{X}$ be a good compactification relative to $S$. Let $s\in S(k)$ 
such that $\overline{X}_s$ intersects the smooth locus of all components of $\overline{X}\backslash X$ non-trivially and such that $X_s\subset\overline{X}_s$ is a partial good compactification, then for every $E\in\Strat(X,\overline{X})^{\rs}$, we have that the natural inclusion
\[(f^*f_*E^{\nabla_{X/S}})_s\to f_s^*f_{s,*} E_s^{\nabla_{X_s/k(s)}}\]
is an isomorphism.
\end{proposition}
\begin{proof}
First let us prove that there is such a natural inclusion: on the left hand side we have the restriction to $X_s$ of a stratified bundle coming from $S$. In particular, it is certainly trivial. On the right hand side, by \cite[Lemma~3.3]{Phung}, we have the maximal trivial sub-stratified bundle of $E_s$, hence the inclusion.

Let $\bar{f}:\overline{X}\to S$ denote the compactification of $f$ and let $E\in\Strat^{\rs}(X,\overline{X})$. As $f_*E^{\nabla_{X/S}}\in\Strat(X)$ it follows that $\bar{f}^*f_*E^{\nabla_{X/S}}\in\Strat(\overline{X})$, in particular $f^*f_*E^{\nabla_{X/S}}$ extends to a $\D_{\overline{X}/S}$-module. By Lemma~\ref{lem:max}, this means that
\[f^*f_*E^{\nabla_{X/S}}\subset E_m,\]
where $E_m$ is defined in the previous section as the maximal subobject of $E$ extending to a $\D_{\overline{X}/S}$-module and clearly $f^*f_*E^{\nabla_{X/S}}=f^*f_*E_m^{\nabla_{X/S}}$.
Similarly, 
\[f_s^*f_{s,*}^{} E_s^{\nabla_{X_s/k(s)}}=f_s^*f_{s,*}^{} ((E_s)_m)^{\nabla_{X_s/k(s)}}\subset (E_s)_m,\] and as by Proposition~\ref{prop:max}, we have that $(E_m)_s=(E_s)_m$, we can without loss of generality assume that $E=E_m$ and hence that $E$ extends to a $\D_{\overline{X}/S}$-module. By Proposition~\ref{prop:prop} $(ii)$ if $\bar{E}$ is the extension of $E$ to $\overline{X}$ then $f^*f_*E^{\nabla_{X/S}}$ is the restriction to $X$ of $\bar{f}^*\bar{f}_*\bar{E}^{\nabla_{\overline{X}/S}}$, and as $\bar{f}$ is proper
\[(\bar{f}^*\bar{f}_*\bar{E}^{\nabla_{\overline{X}/S}})_s=\bar{f}_s^*\bar{f}_{s,*}\bar{E}_s^{\nabla_{\overline{X}_s/k(s)}}\]
by \cite[Lemma~3.3]{Phung} and this concludes the proof.
\end{proof}
\section{Behaviour of the stratified fundamental group under \'etale covers} \label{sec:galois}
In this section we will study the behaviour of the  stratified fundamental group under (generic) Galois covers, whose analogue in characteristic zero is due to Katz (\cite[§1.4]{Katz/Calculation}). 
\begin{lemma}\label{lem:galois}
Let $f:X\to Y$ be a Galois \'etale cover and $E\in\Strat(X)$, let $x\in X(k)$ and $y=f(x)$. Then the monodromy group of $E$ sits in an exact sequence
\[1\to\pi(f^*E,x)\to\pi(E,y)\to \pi(\mathcal{E}_f,y)\to 1\]
where $\mathcal{E}_f$ is the full subcategory of $\langle E\rangle_\otimes$ consisting of all objects that are trivialized by $f$. 
In particular, if $k$ is algebraically closed and $f$ is \'etale with Galois group $G$ we have exact sequences
\[1\to\pi^{\Strat}(Y,y)\to\pi^{\Strat}(X,x)\to G\to 1\]
and 
\[(1\to)\pi^{\rs}(Y,y)\to\pi^{\rs}(X,x)\to G^{t}\to 1\]
where $G^{t}$ is the maximal quotient of $G$ corresponding to an tame cover.
\end{lemma}
\begin{proof}
The second part of the lemma follows from a limit argument on the first exact sequence, together with \cite[Prop.~2.8]{Kindler/FiniteBundles} and the fact if $f$ is \'etale and $F\in\Strat(Y)$, we have that $f_*F\in\Strat(X)$, in particular $F$ is a quotient of $f^*f_*F$ and by Proposition~\ref{prop:crit} $(ii)$ we have the injectivity of the first map. The same holds for $\pi^{\rs}$ under the condition that $G=G^{t}$ by \cite[Cor.~5.6]{Kindler/FiniteBundles}.
So we are left to prove the exactness of the first sequence, where the injectivity of the first map and the surjectivity of the last follows by Proposition~\ref{prop:crit} $(i)$ and $(ii)$. In order to prove this lemma we will make use of the exactness criterion in Proposition~\ref{prop:crit} $(iii)$: point $(a)$ is satisfied by the very definition of $\mathcal{E}_f$ and point $(c)$ by the fact that $F$ is a quotient of $f^*f_*F$ (actually even a subobject). For point $(b)$, let $E\in\Strat(X)$ and let $(f^*E)^0\subset f^*E$ be the maximal trivial subobject. There is a natural action of $G$ on $f^*E=E\otimes_{\Oh_Y}\Oh_X$, given by the fact that $X\to Y$ is a $G$-torsor. This action respects the $\D_{X/k}$-module structure, in particular for every $\sigma\in G$ and every $E'\subset E$ we have that $\sigma(E')$ and $E'$ are isomorphic as $\D_{X/k}$-modules. But then by maximality $(f^*E)^0$ must be invariant under the $G$-action, hence by descent along torsors (see for example \cite[14.85]{Goertz}) $(f^*E)^0$ descends to $Y$ as a module, namely there exists $F\subset E$ such that $f^*F=(f^*E)^0$. To conclude, we must prove that $F$ is actually a sub-stratified bundle of $E$, but it is enough to check this after pullback to $X$, and this completes the proof.
\end{proof}

\begin{corollary}\label{cor:gal}
Let $f:X\to Y$ be a Galois alteration of Galois group $G$, generically \'etale. Then the sequences
\[\pi^{\Strat}(X)\to\pi^{\Strat}(Y)\to \pi(\mathcal{S}_f)\to 1\]
and
\[\pi^{\rs}(X)\to\pi^{\rs}(Y)\to \pi(\mathcal{S}^{\rs}_f)\to 1\]

is exact, where $\mathcal{S}_f$, respectively $\mathcal{S}_f^{\rs}$, is the full subcategory of $\Strat(X/k)$, respectively, of $\Strat^{\rs}(X/k)$, of objects that are trivialized by $f$.
\end{corollary}
\begin{proof}
By Proposition~\ref{prop:prop} $(iii)$ the sequence
\[1\to\pi(f^*E)\to\pi(E)\to \pi(\mathcal{E}_f)\to 1\]
is exact, because it is so when restricted to the \'etale locus of $f$ by the previous Lemma, and we can conclude by a limit argument on all objects in $\Strat(X/k)$, respectively in $\Strat^{\rs}(X/k)$.
\end{proof}
\section{K\"unneth formula for $\pi^{\rs}$}\label{sec:Kun}
In this section we prove the main technical proposition of the article, namely that the regular singular fundamental group satisfies a K\"unneth formula for the product of any two smooth varieties.

\begin{proposition}\label{prop:Kun}
Let $X$ and $Y$ two smooth connected varieties over an algebraically closed field $k$ and $x\in X(k)$, $y\in Y(k)$. Then 
\[\pi^{\rs}(X\times Y,x\times y)\simeq \pi^{\rs}(X,x)\times\pi^{\rs}(Y,y).\] 	
\end{proposition}
\begin{proof}
The projections induce a natural surjective morphism
\[ \pi^{\rs}(X\times Y,x\times y)\to \pi^{\rs}(Y,y)\times \pi^{\rs}(X,x)\]
which we want to prove is an isomorphism. We will drop the base points for notation simplicity.

It is enough to prove that the sequence
\[\pi^{\rs}(X)\to\pi^{\rs}(X\times Y)\to\pi^{\rs}(Y)\]
is exact, and for this we will make use of the exactness criterion stated in Proposition~\ref{prop:crit}$(iii)$. Note that $(c)$ is clear, so we only need to prove $(a)$ and $(b)$. Assume first that $X$ admits a good compactification $\overline{X}$, then $(a)$ and $(b)$ both follow from Proposition~\ref{prop:bc}, if one additionally remarks that if $g:X\times Y\to Y$ denotes the projection then $g^* g_*E^{\nabla_{X\times Y/Y}}$ is regular singular if and only if $g_*E^{\nabla_{X\times Y/Y}}$ is.

Now to the general case.  By de Jong's Galois alteration theorem (\cite{deJong/Curves}), there exists a Galois alteration $f:X'\to X$ of group $G$, admitting a good compactification, in particular over $X'$ Proposition~\ref{prop:bc} holds. Now we can decompose $f$ into $X'\to X''\to X$, where $X''\to X$ is generically \'etale Galois cover and $K(X')/K(X'')$ is purely inseparable, in particularly $X'\to X''$ is generically an universal homeomorphism. Using the criterion in  \cite[Lemma~5.2]{KindlerBattiston} we can reduce to simply check exactness at the level of monodromy groups of single stratified bundles. Using Proposition~\ref{prop:prop}$(ii)$ we can shrink $X'$ and without loss of generality  assume that $X'\to X''$ is a universal homeomorphism. As the category of regular singular stratified bundles is topologically invariant by \cite[Thm.~5.5]{Kindler/Evidence}, if the theorem holds for $X'$ it must hold for $X''$ as well, hence without loss of generality we can assume that $X'\to X$ is separable Galois.

Consider now the subcategory $\mathcal{S}_f^{\rs}\subset\Strat^{\rs}(X)$ of objects that are trivialized by $f$ and the same for $\mathcal{S}_{\tilde{f}}^{\rs}\subset\Strat^{\rs}(X\times Y)$. After restricting these subcategory to the \'etale locus of $f$ and $\tilde{f}$, by \cite[Cor.~2.16]{Kindler/FiniteBundles} we can identify them with full subcategory of $\langle f_*\Oh_{X'}\rangle_\otimes$, respectively of $\langle \tilde{f}_*\Oh_{X'\times Y}\rangle_\otimes$. Again by  \cite[Cor.~2.16]{Kindler/FiniteBundles}, the pullback functor 
\[\langle f_*\Oh_{X'}\rangle_\otimes\to\langle \tilde{f}_*\Oh_{X'\times Y}\rangle_\otimes\]
is an equivalence of categories, as both $f$ and $\tilde{f}$ are Galois covers of group $G$.
We can be more explicit: $\mathcal{S}_f^{\rs}$ is the full subcategory of $\langle f_*\Oh_{X'}\rangle_\otimes$ of all regular singular objects that are defined on the whole $X$ (and not just on the \'etale locus of $f$) and similarly for $\mathcal{S}_{\tilde{f}}^{\rs}$. But then the pullback functor induces an equivalence of categories between them. Let $G'=\pi(\mathcal{S}_f^{\rs},x)=\pi(\mathcal{S}^{\rs}_{\tilde{f}},x\times y)$, then Corollary~\ref{cor:gal} implies that the commutative diagram
\[\xymatrix{
1\ar[r]&\pi^{\rs}(X')\ar[d]\ar[r]&\pi^{\rs}(X'\times Y)\ar[r]\ar[d]&\pi^{\rs}(Y)\ar[r]\ar[d]^{\id}&1\\
1\ar[r]&\pi^{\rs}(X)\ar[d]\ar[r]&\pi^{\rs}(X\times Y)\ar[r]\ar[d]&\pi^{\rs}(Y)\ar[r]\ar[d]&1\\
1\ar[r]&G'\ar[r]^{\id}\ar[d]& G'\ar[r]\ar[d]& 1\\
&1 &1
}\]
has all rows and columns exact except maybe for the second row. Then, a simple diagram chasing show that the middle row needs to be exact as well, which concludes the proof.
\end{proof}

\begin{remark} The previous proof also gives a proof of the K\"unneth formula for the tame fundamental group which is independent from resolution of singularities. If $X$ and $Y$ admit a good compactification, such a proof can be found in \cite{Hoshi}.
\end{remark}
\section{Gieseker's conjecture for homogeneous spaces}\label{sec:conj}
\subsection{The commutativity of $\pi^{\rs}$ for homogeneous spaces}
The results in this section all apply to the tame fundamental group as well, generalizing some  of the results in \cite{BrionSzamuely}.
\begin{proposition}\label{prop:comm}
 Let $X$ be a smooth connected group scheme over a field $k$ and $x\in X(k)$, then $\pi^{\rs}(X,x)$ is commutative. 
\end{proposition}
\begin{proof}
Note that there exists always a rational point due to the existence of the identity section.
Let $\bar{k}$ be an algebraically closed field extension of $k$, then $\pi^{\rs}(X\times_{\Spec k}\Spec\bar{k})\to \pi^{\rs}(X)\times_{\Spec k}\Spec\bar{k}$ is faithfully flat (see \cite[Lemma~4.8]{Battiston/Monodromy} together with \cite[Prop.~3.7]{Battiston/Monodromy} ). Hence we can without loss of generality assume that $k$ is algebraically closed.
The proof follows a classical argument which is originally due, as far as the author knows, to Grothendieck: by Proposition \ref{prop:Kun}, $\pi^{\rs}(X\times X)\simeq\pi^{\rs}(X)\times\pi^{\rs}(X)$, in particular as $X$ is a group object in the category of $k$-schemes, $\pi^{\rs}(X)$ is a group object in the category of $k$-group schemes, and thus is a commutative group scheme (this is well known for abstract groups and naturally extends to group schemes simply by seeing them as group-valued functors).
\end{proof}
In order to prove a similar theorem for homogenous spaces, we introduce a more general class of morphism, namely those who are fppf local product of two varieties:
\begin{definition} Let $f:X\to Y$ be a scheme of $k$-schemes of finite type. We say that $f$ is \emph{fppf-split} if there is a fppf-cover $U\to Y$ and a $k$-scheme $Z$ such that $X\times_Y U\simeq Z\times_{\Spec k} U\xrightarrow{\pi} U$ where the map is the projection to the same factor. We say that $f$ is \emph{generically} fppf-split if there exists a dense Zariski open $V\subset Y$ such that $X\times_YV\to V$ is fppf-split.
\end{definition}

\begin{example} The main example to keep in mind is when $f:X\to Y$ is (generically) a $H$-torsor for some (smooth connected) group $H$. Note that in this situation the fppf cover $U\to Y$ can be taken to be $f$ itself (see for example \cite[Tag 04TZ]{stacks-project}).
\end{example}

\begin{lemma}\label{lem:homsp}
Let $X$, $Y$ be two smooth connected varieties over and $f:X\to Y$ be any generically fppf-splitmap with smooth geometrically connected fibers. Let $x\in X(k)$ and $y=f(x)$, then the induced map  
\[\pi^{\rs}(X,x)\to\pi^{\rs}(Y,y)\]
 is faithfully flat. In particular, if $X=G$ is a smooth connected group and $Y=G/H$ is a homogeneous space with $H\subset G$ a connected subgroup, then $\pi^{\rs}(Y)$ is commutative.
\end{lemma}
\begin{proof}We drop the base points from the notations in the proof.
As in the previous proposition, we can again base change to the algebraic closure of $k$ without loss of generality. Moreover if $U$ is any dense open in $Y$, the induced morphism $\pi^{\rs}(U)\to\pi^{\rs}(Y)$ is faithfully flat, hence we can without loss of generality assume that $f$ is fppf-split.

The rest of the proof is a repeated use of Proposition~\ref{prop:Kun} and fppf-descent: the fibers are smooth and connected, and if $U\to Y$ is a fppf-cover such that $X\times_Y U\simeq Z\times U\to U$, then $Z$ is smooth and connected hence, by fppf descent (\cite[Tag 02VL]{stacks-project}), so is $f$, in particular  $f$ is fppf.

In order to prove that $\pi^{\rs}(X)\to\pi^{\rs}(Y)$ is faithfully flat, by Proposition~\ref{prop:crit} $(i)$ we need to show that the pullback map $f^*:\Strat^{\rs}(Y)\to\Strat^{\rs}(X)$ is fully faithful and that for every $E\in\Strat^{\rs}(Y)$ and $F\subset f^*E$, there exists $F'\subset E$ such that $F=f^*F'$.
We start with the latter: let $F\subset f^*E$, then it is enough to prove that $F$ descends as coherent sheaf, as checking its invariance under the $\mathcal{D}_{Y/k}$-action can be done after faithfully flat base change. As the restriction functor $\langle E\rangle_{\otimes}\to\langle E_{|V}\rangle_\otimes$ to any dense open $V$ is an equivalence of categories (see Proposition~\ref{prop:prop} $(ii)$), we can work with any dense open of any connected component of $U$ and its (open) image in $Y$ in particular we can assume that $U$ is connected. Let $U^{red}$ be the underlying reduced subscheme of $U$, then as $U^{red}\to Y$ is again of finite presentation and quasi compact, by generically flatness (\cite[Thm.~14.5]{Goertz}) we can without loss of generality assume that $U$ is a connected reduced faithfully flat cover of a dense open $V\subset Y$ trivializing the restriction $f=f_{|U}:X\times_VY=X_V\to V$. By shrinking again we can also achieve that $U$ is smooth.
We will still denote $F$ its restriction to $X_V$ and similarly for $E$.

Now, by Proposition~\ref{prop:Kun}, we have an exact sequence 
\[1\to\pi^{\rs}(Z)\to\pi^{\rs}(X\times_V U\simeq Z\times_{\Spec k} U)\to\pi^{\rs}(U)\to 1\]
in particular we know that after pulling back $F$ along $X\times_V U\to X_V$ it indeed comes from $\tilde{F}\in\Strat^{\rs}(U)$. As $U\to V$ is faithfully flat, hence of effective descent, to prove that $\tilde{F}$ is defined over $U$ it suffices to exhibit a descent data with respect to $U\to V$, that is an isomorphism $p_1^*\tilde{F}\simeq p_2^*\tilde{F}$ respecting the usual cochain condition, where $p_1,p_2:U\times_V U\to U$ are the natural projections. Let $s:U\to Z\times_{\Spec k} U$ be any section (recall that we are assuming that $k$ is algebraically closed), then the $p_i$ sits in the following diagram, where all squares are Cartesian:
\[
\xymatrix{
U\times_V U\ar[r]\ar[d]^{p_1}\ar@/^.9pc/[rrr]^{p_2}&(Z\times U)\times_{X_V}(Z\times U)\ar[d]^{q_1}\ar[r]^(.65){q_2}&Z\times_{\Spec k} U\ar[d]^{\pi_1}\ar[r]_{\pi_2}&U\ar[d]\\
U\ar[r]^s& Z\times_{\Spec k} U\ar[r]& X_V\ar[r]& V
}.\]
Now, $\pi_2^*F'=\pi_1^*\tilde{F}$ comes with natural descent data with respect to $Z\times{\Spec k} U\to X_V$. Pulling back this descent data via the section $s$ we get a descent data for $F'$ with respect to $U\to V$ and by fppf descent, $F'\subset E$ such that $f^*F'=F$.
In order to prove the fully faithfulness of $f^*:\Strat^{\rs}(Y)\to\Strat^{\rs}(X)$ we can use similar arguments: indeed if $E,F\in\Strat^{\rs}(Y)$ and $g:f^*E\to f^*F$ is a morphism of $\D_{X/k}$-modules, it is enough to check that $g$ is defined on $Y$ as morphism of coherent sheaves, and the same machinery of fppf descent applies.

For the last part of the lemma, notice that under these conditions $G/X$ is connected and by \cite[VI, Thm.~3.2]{SGA3} it is also smooth, and the existence of a rational point is automatic, in particular it makes sense to consider its regular singular stratified fundamental group. 
As $\pi^{\rs}$ is topologically invariant (see \cite[Thm.~5.5]{Kindler/Evidence}) we can assume without loss of generality, like in \cite[Sec.~3]{BrionSzamuely}, that $H$ is a reduced, hence smooth, subgroup of $G$.
Moreover by \cite[VI, Thm.~3.2]{SGA3} $f:G\to X$ is a fppf $H$-torsor, hence there exist a fppf cover $U\to X$ such that $id\times f:U\times_X G\to U$ is isomorphic to the projection $U\times_{\Spec k} H\to U$ and  Proposition~\ref{prop:comm} completes the proof.

\end{proof}

\begin{remark}
If $X\simeq G/H$ with $H$ not connected then one cannot in general hope that $H$ is abelian: if $H$ is finite \'etale we have that $X\to G/H$ being an $H$-torsor, is an \'etale Galois cover. In particular by Lemma~\ref{lem:galois} if $H$ moreover corresponds to a tame cover of $G/H$, we get an short exact sequence
\[1\to\pi^{\rs}(G)\to\pi^{\rs}(X)\to H\to 1\]
and if $H$ is not abelian the same must hold for $\pi^{\rs}(X)$. Still, we  will be able show in the proof of Theorem~\ref{thm:conj} that if $X$ is a connected homogenous space, $\pi^{\rs}(X)$ is always an extension of a commutative group scheme by a finite one.
\end{remark}

\subsection{The abelian unipotent part of $\pi^{\rs}$}
We need a few generalizations of the results obtained in \cite{DosSantos} and in \cite{Kindler/Evidence}. Note that when $k$ is algebraically closed, all choices of a closed point give isomorphic Tannaka duals, hence we can omit the base point from the notation.
\begin{proposition}[{compare with \cite[Thm.~15]{DosSantos}}]\label{prop:uni}
Let $X$ be a smooth connected variety over $k$ algebraically closed, then the largest abelian unipotent quotient of $\pi^{\rs}(X)$ is profinite, and correspond to the largest unipotent abelian quotient of $\pi^{t}(X)$.
\end{proposition}
\begin{proof}
It is enough to prove that for every $E\in\Strat^{\rs}(X)$, if $\pi(E)$ is abelian and unipotent it has to be finite. For this, it is enough to show it does not have any quotient isomorphic to $\mathbb{G}_a$ or, equivalently, that there is no $E\in\Strat^{\rs}(X)$ with $\pi(E)=\mathbb{G}_a$. By way of contradiction assume that there is such $E$, then as $\mathbb{G}_a$ admits a faithful two dimensional representation which is an extension of the trivial representation, we can assume without loss of generality that $E\in\Ext^1_{\Strat^{\rs}}(\mathbb{I},\mathbb{I})$. 
If $X$ admits a good compactification $\overline{X}$, the proof of  \cite[Lemma~4.1]{Kindler/Evidence} shows that $E$ extends to $\bar{E}\in\Strat(\overline{X})$ and hence by \cite[Thm.~15]{DosSantos} $\pi(E)=\pi(\bar{E})$ is finite.
In general, let $f:Y\to X$ be a separable Galois modification of $X$ provided by de Jong's theorem and let $\mathcal{S}\subset\langle E\rangle_\otimes$ be the full subcategory of objects that are trivialized by $f^*$. As in the proof of Proposition~\ref{prop:Kun}, by topological invariance we can assume that $f$ is Galois. Then Lemma~\ref{lem:galois} implies that the short sequence
\[1\to\pi(f^*E)\to\pi(E)\to G'\to 1\]
is exact, with $G'=\pi(\mathcal{S})$ finite. In particular, as $f^*E\in\Ext^1_{\Strat^{\rs}(Y)}(\mathbb{I},\mathbb{I})$, by what we just proved $\pi(f^*E)$ is finite, hence so must be $\pi(E)$.
\end{proof}
The next proposition is proven in \cite[Thm.~4.2]{Kindler/Evidence}, assuming the existence of a good compactification for $X$:
\begin{proposition}[{compare with \cite[Thm.~4.2]{Kindler/Evidence}}]\label{prop:larsabe}
Let $X$ be a smooth connected variety over an algebraically closed field $k$. If $\pi^{t}(X)$ has trivial abelianization, the same holds for $\pi^{\rs}(X)$.
\end{proposition}
\begin{proof}
Let $E\in\Strat^{\rs}(X)$ such that $\pi(E)$ is commutative, then $\pi(E)$ decomposes as a direct product of unipotent and multiplicative part. The first is trivial thanks to Proposition~\ref{prop:uni}, while the latter is trivial by \cite[Prop.~3.16]{Kindler/Evidence}.
\end{proof}

\begin{proposition}[{compare with \cite[Cor.~1.6]{Kindler/Evidence}}]\label{prop:abemor} Let $X$, $Y$ be smooth proper connected  varieties over an algebraically closed field $k$. Let $f:X\to Y$ such that the induced map $\pi^{t,ab}(X)\to\pi^{t,ab}(Y)$ is an isomorphism. Then $\pi^{\rs,ab}(X)\to\pi^{\rs,ab}(Y)$ is an isomorphism.
\end{proposition}
\begin{proof}
This follows immediately from \cite[Cor.~1.6]{Kindler/Evidence} and \cite[Thm.~15]{DosSantos}.
\end{proof}
\subsection{Proof of the conjecture for homogeneous spaces}
\begin{definition} Let $X$ be a separated scheme of finite type over $k$, then we say that $X$ is \emph{quasi homogeneous} if it has a dense open subscheme $U$ which is a homogeneous space.
\end{definition}
\begin{theorem}\label{thm:conj} Let $X$ be a smooth connected  quasi homogeneous space such that $X_{\bar{k}}$ has trivial tame fundamental group. Then $\pi^{\rs}(X)=0$.
\end{theorem}
\begin{proof}
By \cite[Lemma~4.8]{Battiston/Monodromy} together with \cite[Prop.~3.7]{Battiston/Monodromy} we can without loss of generality assume that $k$ is algebraically closed. 
Assume first that $X$ is an homogeneous space. Therefore there exists $G$ group scheme over $k$ such that $X$ is (isomorphic to) the fppf quotient $G/H$ for some subgroup $H\subset G$. Let $G_0$ be the connected component of the identity in $G$, then the composition $G_0\to G\to G/H$ has kernel $H\cap G_0$, in particular we have a monomorphism $j:G_0/(G_0\cap H)\to G/H$, that is a closed embedding. As 
\[\dim G/H=\dim G-\dim H=\dim G_0-\dim (G_0\cap H)=\dim G_0/(G_0\cap H),\]
and $X=G/H$ is smooth, we have that $G_0/(G_0\cap H)$ is isomorphic to a connected component of $G/H$, but the latter is connected hence the closed embedding $j$ is actually an isomorphism. That is, without loss of generality we can assume that $G$ is connected. It is moreover smooth by \cite[VI, Thm.~3.2]{SGA3}.
 By topological invariance of $\pi^{\rs}$ we can assume as in the proof of Lemma~\ref{lem:homsp} that $H$ is reduced, hence smooth. There is hence a canonical short exact sequence
\[1\to H_0\to H\to H^{\et}\to 1,\]
where $H_0$ is the connected component of the identity of $H$ and $H^{\et}$ is an \'etale group. In particular we can consider the induced faithfully flat morphism
\[G/H_0\to G/H=(G/H_0)/ H^{\et},\]
which will be a $H^{\et}$-torsor by \cite[VI, Thm.~3.2]{SGA3}, that is, an \'etale cover. 
In particular by Lemma~\ref{lem:galois} we have an exact sequence
\[\pi^{\rs}(G/H_0)\to\pi^{\rs}(G/H)\to H^{t}\to 1,\]
where $H^t$ is the maximal quotient of $H^{\et}$ corresponding to a tame cover of $G/H$.

As $G/H$ has no non-trivial tame covers, we have that $H^{t}$ is trivial, hence $\pi^{\rs}(G/H)$ is a quotient of $\pi^{\rs}(G/H_0)$, which is commutative by Lemma~\ref{lem:homsp} and Proposition~\ref{prop:larsabe} concludes the proof.

Now to the general case: there exists a dense $U\subset X$, with $U$ a homogeneous space. By the first part of the proof, we have an exact sequence
\[A\to \pi^{\rs}(U)\to H\to 1\]
where $A$ is a commutative group and $H$ is a finite group. It is enough to show that the composition $\phi:A\to\pi^{\rs}(U)\to\pi^{\rs}(X)$ is surjective, and then the proof follows form Proposition~\ref{prop:larsabe}.
The image of $A$ in $\pi^{\rs}(U)$ has finite index and is normal, in particular by Proposition~\ref{prop:prop} $(ii)$, $\phi(A)$ is normal and of finite index in $\pi^{\rs}(X)$, but as $\pi^t(X)$ is trivial, there are no non-trivial finite quotients of $\pi^{\rs}(X)$, hence $\phi(A)=\pi^{\rs}(X)$.
\end{proof}

If everything is proper, we are now able to prove the relative form of Gieseker conjecture:
\begin{proposition}\label{prop:giemor} Let $X,Y$ be smooth proper homogeneous varieties over an algebraically closed field $k$ and $f:X\to Y$ such that  $f^*:\pi^{t}(X)\to\pi^{t}(Y)$ is an isomorphism. Then  $f^*:\pi^{\rs}(X)\to\pi^{\rs}(Y)$ is an isomorphism.
\end{proposition}
\begin{proof}
By the proof of Theorem~\ref{thm:conj}, there are affine commutative group schemes $A_X$ and $A_Y$, finite \'etale group schemes $H_X$ and $H_Y$ and exact sequences
\[\begin{split}A_X\to \pi^{\rs}(X)\to H_X\to 1,\\
A_Y\to \pi^{\rs}(Y)\to H_Y\to 1.\end{split}\]
As $f$ induces an isomorphism of the tame fundamental groups, we can assume without loss of generality that $H_X=H_Y$, and consider the corresponding tame \'etale covering $Y'\to Y$ and $X'=Y'\times_Y X\to X$. 
By Lemma~\ref{lem:galois} together with the fact that $H_X=H_Y$ correspond to tame coverings of $X$ and $Y$, we get a commutative diagram with exact rows
\[\xymatrix{
1\ar[r]&\pi^{\rs}(X')\ar[r]\ar[d]&\pi^{\rs}(X)\ar[d]\ar[r]&H_X\ar[d]^{\id}\ar[r]& 1\\
1\ar[r]&\pi^{\rs}(Y')\ar[r]&\pi^{\rs}(Y)\ar[r]&H_Y\ar[r]& 1
}.\]
The same diagram holds for the tame fundamental group, therefore the isomorphism $\pi^{t}(X)\to \pi^{t}(Y)$ restricts to an isomorphism $\pi^{t}(X')\to\pi^{t}(Y')$ and the latter are abelian as they are in the image of $A_X$ and $A_Y$ respectively. Hence Proposition~\ref{prop:abemor} implies that $\pi^{\rs}(X')\to\pi^{\rs}(Y')$ is an isomorphism and so is $\pi^{\rs}(X)\to\pi^{\rs}(Y)$.
\end{proof}

\addcontentsline{toc}{section}{\refname}
\printbibliography

\end{document}